\documentclass{amsart}
\usepackage{amssymb}

\usepackage{color}
\usepackage{graphics}
\usepackage{amsmath,amscd}
\usepackage{amsthm}
\usepackage[all,cmtip]{xy}
\usepackage{tikz}
\usepackage{graphicx}
\usepackage{epstopdf}
\usepackage{mathtools}
\theoremstyle{definition}

\usepackage[colorlinks=green,linkcolor=blue,citecolor=red,urlcolor=red]{hyperref}

\newcommand{\Z}{\mathbb{Z}}
\newcommand{\C}{\mathbb{C}}

\newtheorem{proposition}{Proposition}[section]
\newtheorem{theorem}[proposition]{Theorem}
\newtheorem{definition}[proposition]{Definition}
\newtheorem{lemma}[proposition]{Lemma}

\newtheorem{corollary}[proposition]{Corollary}
\newtheorem{remark}[proposition]{Remark}

\begin{document}

\title{Contact and isocontact embedding of $\pi$--manifolds}


\thanks{}

\author{Kuldeep Saha}

\address{Chennai Mathematical Institute,
H1, SIPCOT IT 	Park, Siruseri, Kelambakkam
603103, Tamilnadu, India}
\email{kuldeep@cmi.ac.in}



\begin{abstract}
	
 We prove some contact analogs of smooth embedding theorems for closed $\pi$--manifolds. We show that a closed, $k$-connected, $\pi$--manifold of dimension $(2n+1)$ that bounds a $\pi$--manifold, contact embeds in the $(4n-2k+3)$-dimensional Euclidean space with the standard contact structure. We also prove some isocontact embedding results for $\pi$--manifolds and parallelizable manifolds.

\end{abstract}

\maketitle

\section{introduction}

Embedding of manifolds in the Euclidean spaces has long been a problem of great importance in geometric topology. Over the years many remarkable results concerning embedding of manifolds were obtained. The first major breakthrough in this direction was the Whitney embedding theorem~\cite{Wh}, which says that every $n$-manifold can be smoothly embedded in $\mathbb{R}^{2n}$. Later, Haefliger and Hirsch (\cite{HH}) generalized Whitney's theorem to show that a closed, orientable, $k$-connected $n$-manifold can be embedded in $\mathbb{R}^{2n-k-1}$. 

Recall that a manifold $M$ is called a $\pi$-manifold, provided the direct sum of its tangent bundle with the trivial real line bundle is trivial. In \cite{Sa}, Sapio improved the Haefliger-Hirsch  embedding theorem for $k$-connected $\pi$-manifolds to produce embeddings with trivial normal bundle in $\mathbb{R}^{2n-2k-1}$. In this note we study the contact analogs of some of these embedding results.

 A contact manifold is an odd dimensional smooth manifold $M^{2n+1}$, together with a maximally non-integrable hyperplane distribution $\xi \subset TM$. A contact form $\alpha$ representing $\xi$ is a local $1$-form on $M$ such that $\xi=Ker\{\alpha\}$. The contact condition is equivalent to saying that $\alpha\wedge(d\alpha)^n$ is a volume form. The $2$-form $d\alpha$ then induces a conformal symplectic structure on $\xi$. If the line bundle $TM/\xi$ over $M$ is trivial, then the contact structure is said to be co-orientable. For a co-orientable contact structure $\xi$, one can define a contact form $\alpha$ representing $\xi$ on all of $M$. In this article, we will only consider co-orientable contact structures on closed, orientable manifolds. We will denote a manifold $M$ together with a contact structure $\xi$ by $(M,\xi)$. We will use $\xi_{std}$ to denote the standard contact structure on an odd dimensional Euclidean space $\mathbb{R}^{2N+1}$ given by $Ker\{dz + \Sigma^N_{i=1} x_idy_i\}$. The trivial real vector bundle of rank $r$ over a space $Z$ will be denoted by $\varepsilon^r_Z$ and $\varepsilon^r_Z(\mathbb{C})$ will denote the trivial complex vector bundle of complex rank $r$ over $Z$. An embedding is always assumed to be smooth. All the embedding results will be stated for closed manifolds.

\begin{definition}[Isocontact embedding] $(M^{2n+1},\xi)$ admits an isocontact embedding in $(V^{2N+1},\eta)$, if there is an embedding $\iota: M \hookrightarrow V$ such that for all $p$ in M, $D\iota(T_pM)$ is transverse to $\eta_{\iota(p)}$ and $D\iota(T_pM)\cap{\eta}_{\iota(p)} = D\iota({\xi}_p)$. A manifold $M^{2n+1}$ contact embeds in $(V^{2N+1},\eta)$ if there exists a contact structure $\xi_0$ on $M^{2n+1}$ such that $(M,\xi_0)$ has an isocontact embedding in $(V^{2N+1},\eta)$.
\end{definition}	
	
It follows from the definition that if $\alpha$ is a contact form representing $\xi$ and $\beta$ is a contact form representing $\eta$. Then $\iota^*(\beta) = h\cdot\alpha$ for some positive function $h$ on $M$. $D\iota(\xi)$ is a conformal symplectic sub-bundle of $(\eta|_{\iota(M)},d\beta)$.

Similarly, we can define the notion of an isocontact immersion.
	  
\begin{definition}[Isocontact immersion]
		An isocontact immersion of $(M,\xi)$ in $(\mathbb{R}^{2N+1},\eta)$ is an immersion $\iota:(M,\xi)\looparrowright (\mathbb{R}^{2N+1},\eta)$ such that $D\iota(TM)$ is transverse to $\eta$ and $D\iota(TM)\cap{\eta} = D\iota(\xi)$.
\end{definition}
      
 Gromov \cite{Gr} reduced the existence of an isocontact embedding of a contact manifold $(M^{2n+1},\xi)$ in a contact manifold $(V^{2N+1},\eta)$, for $N \geq n+2$, to a problem in obstruction theory. Gromov \cite{Gr} proved that any contact manifold $(M^{2n+1},\xi)$ has an isocontact embedding in $(\mathbb{R}^{4n+3},\xi_{std})$. This result, which is essentially the contact analog of Whitney's embedding theorem, was reproved later by A. Mori \cite{Mo} for $n=1$ and by D. M. Torres \cite{Tor} for all $n$ using different techniques. For isocontact embeddings of a contact manifold $(M,\xi)$ of co-dimension $\leq dim(M)-1$, there is a condition on the Chern classes of $\xi$. This condition comes from the normal bundle of the embedding. See the Remark \ref{remark on chern class and normal bundle} for a precise statement. So, one has to restrict the isocontact embedding question to contact structures which satisfy that condition. Given this, a theorem of N. Kasuya (\cite{Ka}, Theorem $1.5$) says that for $2$-connected $(2n+1)$-contact manifolds, the Haefliger-Hirsch theorem has a contact analog giving isocontact embedding in $(\mathbb{R}^{4n+1},\xi_{std})$. Here, we investigate the contact analog of Sapio's theorem for $\pi$-manifolds and also provide some contact embedding results for parallelizable manifolds. Before stating our results, we introduce some terminologies.   

\

In \cite{Sa}, Sapio introduced the notion of an almost embedding. A manifold $M^n$ \emph{almost embeds} in a manifold $W^N$, if there exists a homotopy sphere $\Sigma^n$ so that $M^n\#\Sigma^n$ smoothly embeds in $W^N$. We want to define analogous notions for contact and isocontact embeddings. Recall that if $(M,\xi_M)$ and $(N^{2n+1},\xi_N)$ are two contact manifolds, then by $(M^{2n+1}\#N^{2n+1},\xi_M\#\xi_N)$ we denote the contact connected sum of them. For details on a contact connected sum we refer to chapter $6$ of \cite{Ge}.
 
 \begin{definition}[Homotopy isocontact embedding]
 	$(M^{2n+1},\xi)$ admits a homotopy isocontact embedding in $(\mathbb{R}^{2N+1},\xi_{std})$, if there exists a contact homotopy sphere $(\Sigma^{2n+1},\eta)$ such that $(M^{2n+1}\#\Sigma^{2n+1},\xi\#\eta)$ has an isocontact embedding in $(\mathbb{R}^{2N+1},\xi_{std})$. We say, $M^{2n+1}$ homotopy contact embeds in $(\mathbb{R}^{2N+1},\xi_{std})$, if there is a contact structure $\xi_{0}$ on $M\#\Sigma^{2n+1}$ such that $(M\#\Sigma^{2n+1},\xi_{0})$ has an isocontact embedding in $(\mathbb{R}^{2N+1},\xi_{std})$.  
 \end{definition}

Before stating our results, we describe the contact structures we will be considering. For an isocontact embedding of $(M^{2n+1},\xi)$ of co-dimension $2(N-n)$ with trivial symplectic normal bundle, we need that the Chern classes $c_i(\xi)$ vanish for $1 \leq i \leq n$. For details see the Remark \ref{remark on chern class and normal bundle}. Note that by a theorem of Peterson (Theorem $2.1$, \cite{Ke}), if $(M^{2n+1},\xi)$ is torsion free, then this condition is true if and only if $\xi$ is trivial as a complex vector bundle over the $2n$-skeleton of $M$. Consider the fibration map $SO(2n+2) \rightarrow \Gamma_{n+1}$ with fiber $U(n+1)$, where $\Gamma_{n+1}$ denotes the space of almost complex structures on $\mathbb{R}^{2n+2}$. Since $TM \oplus \varepsilon_M^1$ is trivial for a $\pi$--manifold, one can postcompose a trivialization map to $SO(2n+2)$ with the above fibration map to get an almost contact structure on $M$. For notions of almost complex and almost contact structures see section $2.2$. 

\begin{definition}
	A contact structure on a $\pi$--manifold $M$, representing an almost contact structure that factors through a map from $M$ to $SO(2n+2)$ as mentioned above, will be called an \emph{$SO$-contact structure}.
\end{definition}

 We now state an analog of Sapio's Theorem for contact $\pi$--manifolds.

\begin{theorem} \label{1st theorem}
	
	Let $M^{2n+1}$ be a $k$-connected, $\pi$--manifold. Assume that $n \geq 2$ and $k \leq n-1 $. Then
	
	\begin{enumerate}
		
		\item $M^{2n+1}$ homotopy contact embeds in $({\mathbb{R}}^{4n-2k+3},\xi_{std})$.
		\item If $n \not \equiv 3 \ (mod \ 4)$ and for all $i \in \{k+1,\cdots,2n-k\}$ such that $i \equiv 0,2,6,7 \ (mod \ 8),\ H_{2n-i+1}(M) = 0$, then for any contact structure $\xi$ on $M^{2n+1}$, $(M,\xi)$ has a homotopy isocontact embedding in $({\mathbb{R}}^{4n-2k+3},\xi_{std})$. 
		\item If $n \not \equiv 3 \ (mod \ 4)$ and for all $i \in \{k+1,\cdots,2n-k\}$ such that $i \equiv 0,7 \ (mod \ 8),\ H_{2n-i+1}(M) = 0$, then for any $SO$-contact structure $\xi$ on $M^{2n+1}$, $(M,\xi)$ has a homotopy isocontact embedding in $({\mathbb{R}}^{4n-2k+3},\xi_{std})$.  
		\item If $M^{2n+1}$ bounds a $\pi$--manifold, then we can omit ``homotopy'' in the above statements.
		
	\end{enumerate}

We remark that in all the statements above, we get contact or isocontact embeddings with a trivial conformal symplectic normal bundle.

\end{theorem}

Note that Theorem \ref{1st theorem} provides criteria to find examples of isocontact embeddings of $\pi$-manifolds in the standard contact euclidean space. For example, a straightforward application of statement $2$ and $4$ in Theorem \ref{1st theorem} shows that every contact structure on $S^4 \times S^5$ has an isocontact embedding in $(\mathbb{R}^{15},\xi_{std})$. Here, we have $k = 2$ and $n = 4$. Similarly, one can check that all contact structures on $S^4 \times S^9$ ($k = 3$) and $S^{11} \times S^{12}$ ($k = 10$) admit isocontact embeddings in $(\mathbb{R}^{21},\xi_{std})$ and $(\mathbb{R}^{27},\xi_{std})$ respectively.

The proof of Theorem $\ref{1st theorem}$ is based on Gromov's h-principle for existence of contact structure on open manifold (see Theorem \ref{existence of contact structure gromov}). Roughly speaking, we put a contact structure on a tubular neighborhood of the embedded contact manifold, extend it to an almost contact structure on the ambient manifold using  obstruction theory and then apply Gromov's h-principle.

\begin{corollary} \label{1st corollary}
	Let $M^{2n+1}$ be an $(n-1)$-connected $\pi$--manifold that bounds a $\pi$--manifold. Then
	
	\begin{enumerate}
		\item $M^{2n+1}$ contact embeds in $({\mathbb{R}}^{2n+5},\xi_{std})$.
		\item If $n \equiv 4,5 \ (mod\ 8)$, then for any contact structure $\xi$, $(M,\xi)$ has an isocontact embedding in $({\mathbb{R}}^{2n+5},\xi_{std})$.
	\end{enumerate}
	
   	In particular, any contact homotopy sphere ${\Sigma}^{2n+1}$ that bounds a parallelizable manifold has an isocontact embedding in $({\mathbb{R}}^{2n+5},\xi_{std})$, for $n \equiv 0,1,2\ (mod\ 4)$. 
\end{corollary} 

  For example, by \cite{KM}, we get that all $11$-dimensional contact homotopy spheres has an isocontact embedding in $({\mathbb{R}}^{15},\xi_{std})$.

 \begin{remark}{(\emph {On optimal dimension of embedding})}
 In section $4$ of \cite{Sa}, Sapio constructs a family of $(r- \rho(r) - 1)$-connected $(2r - \rho(r) - 1)$-dimensional manifolds, denoted by $M(r)$. Here, $r = (2a+1) 2^{b+4c}, 0 \leq b \leq 3, a,b,c \in \mathbb{Z}_{\geq0}$ and $\rho(r) = 2^b + 8c$. Note that $[2(2r -\rho(r) - 1) - 2(r - \rho(r) - 1) - 1] = 2r - 1$. Sapio shows that $M(r)$ bounds a $\pi$-manifold. So, by Theorem \ref{thm:sapio}, $M(r)$ embeds in $\mathbb{R}^{2r-1}$. But $M(r)$ does not embed in $\mathbb{R}^{2r-2}$. It follows from the discussion in section \ref{codim 4 embedn} that if we assume $n-k \geq 2$ in Theorem \ref{1st theorem}, then we can improve the dimension of embedding to $4n - 2k + 1 = 2(2n + 1) - 2k - 1$. Therefore, the family of examples given by the manifolds $M(r)$, for $\rho(r) \geq 4$, actually show that for $n-k \geq 2$, $(4n - 2k +1)$ is the optimal dimension of contact embedding.
 	
 \end{remark}

 Using similar techniques as in Theorem \ref{1st theorem} and Gromov's h-principles for contact immersion and isocontact embedding (see \ref{h-principle for contact immersion} and \ref{h-principle for contact embedding}) we prove the following result for parallelizable manifolds.

\begin{theorem} \label{2nd theorem}
	Let $M^{2n+1}$ be a parallelizable manifold.
	\begin{enumerate}
		\item For any contact structure $\xi$ on $M^{2n+1}$, $(M^{2n+1},\xi)$ contact immerses in $(\mathbb{R}^{2n+3},\xi_{std})$.  
		\item If $M^{2n+1}$ is $5$-connected, then for $n \equiv 0,1\ (mod \ 4)$ and $n \geq 7$, any contact structure $\xi$, $(M^{2n+1},\xi)$ has an isocontact embedding in $(\mathbb{R}^{4n-3},\xi_{std})$.
	\end{enumerate}   
\end{theorem}

\begin{corollary} \label{2nd corollary}
	Let $M^{2n+1} = N^{2n-1} \times (S^1 \times S^1)$. Where $N^{2n-1}$ is a $\pi$--manifold that embeds in $\mathbb{R}^{2N+1}$ with trivial normal bundle. Then $M^{2n+1}$ contact embeds in $(\mathbb{R}^{2N+5},\xi_{std})$. 
\end{corollary}

In \cite{BEM}, S. Borman, Y. Eliashberg and E. Murphy defined the notion of an overtwisted contact ball in all dimensions. Any contact structure that admits a contact embedding of such an overtwisted ball is called an overtwisted contact structure. These contact structures were shown to satisfy the h-principle for homotopy of contact structures. For details see Theorem $\ref{homotopy of contact structure bem}$. Using this, we prove a uniqueness result for embedding of certain $\pi$-manifolds in an overtwisted contact structure $\eta_{ot}$ on $\mathbb{R}^{2N+1}$, analogous to Theorem $1.25$ in \cite{EF}. 

\begin{theorem}\label{3rd theorem} Let $(M^{8k+3},\xi)$ be a contact $\pi$--manifold such that $H_i(M;\mathbb{Z}) = 0$, for $i \equiv \ 2,4,5,6 \ (mod \ 8)$. Let $\iota_1,\iota_2 : (M^{8k+3},\xi) \rightarrow (\mathbb{R}^{2N+1},\eta_{ot})$ be two isocontact embeddings with trivial conformal symplectic normal bundle such that both the complements of $\iota_1(M)$ and $\iota_2(M)$ in $(\mathbb{R}^{2N+1},\eta_{ot})$ are overtwisted. If $\iota_1$ and $\iota_2$ are smoothly isotopic, then there is a contactomorphism $\chi : (\mathbb{R}^{2N+1},\eta_{ot}) \rightarrow (\mathbb{R}^{2N+1},\eta_{ot})$ such that $\chi \cdot \iota_1 = \iota_2$. 
\end{theorem}

For example, any two isocontact embeddings of $(S^{8k_1} \times S^{8k_2+3},\xi_{0})$ in $(\mathbb{R}^{8k_1+8k_2+5},\eta_{ot})$ which satisfy the hypothesis of Theorem $\ref{3rd theorem}$, are equivalent. 
  
 \
 
We would like to mention that the problem of isocontact embedding of $3$-manifolds in $\mathbb{R}^5$ has seen much development in the past few years. The first result in this direction was given by Kasuya. In \cite{Ka2}, he proved that given a contact $3$-manifold $(M,\xi)$, the first Chern class $c_1(\xi)$ is the only obstruction to isocontact embedding of $(M,\xi)$ in some contact structure on $\mathbb{R}^5$. The approach in \cite{Ka2} was a motivation for the present article. Afterwards, Etnyre and Furukawa \cite{EF} showed that a large class of contact $3$-manifolds embed in $(\mathbb{R}^5,\xi_{std})$. Recently, the existence and uniqueness question for co-dimension $2$ isocontact embedding has been completely answered by the works of   Pancholi and Pandit \cite{PP}, Casals, Pancholi and Presas \cite{CPP}, Casals and Etnyre \cite{CE} and Honda and Huang \cite{HoH}. On the other hand explicit examples of co-dimension $2$ isocontact embeddings were produced in the works of Casals and Murphy \cite{CM}, Etnyre and Lekili \cite{EL} and in \cite{S}.

\subsection{Acknowledgment} The author is grateful to Dishant M. Pancholi for his help and support during this work. He would like to thank Suhas Pandit for reading the first draft of this note and for his helpful comments. He also thanks John Etnyre for clarifying some doubts regarding the proof of Theorem $1.25$ in \cite{EF}. Finally, he thanks the referee for various comments and suggestions which helped improve the article. The author is supported by the National Board of Higher Mathematics, DAE, Govt. of India.
 	   
\section{preliminaries}

In this section we review some basic notions and results that will be relevant to us.

\subsection{h-principle for immersion} 

Let $f,g: M^{2n+1} \looparrowright V^{2N+1}$ be two immersions. We say that $f$ is \textit{regularly homotopic} to $g$, if there is a family $h_t: M \looparrowright V$ of immersions joining $f$ and $g$. Being an immersion, $f$ induces $Df : TM \rightarrow TV$ such that for all $p \in M$, $Df$ restricts to a monomorphism $Df_p$ from $T_pM$ to $T_{f(p)}V$.  

\begin{definition}[Formal immersion]
A formal immersion of $M$ in $V$ is a bundle map $F:TM \rightarrow TV$ that restricts to a monomorphism $F_p$ on each tangent space $T_pM$, for $p\in M$. It can be represented by the following diagram, where $f_0$ is any smooth map making it commutative.

 \begin{equation*}
 \begin{CD}
 TM   @>F>>  TV\\
 @VV\pi_1V        @VV\pi_2V\\
 M     @>f_0>>  V
 \end{CD}
 \end{equation*}

\end{definition} 

We say that $F$ is a formal immersion covering $f_0$. The existence of a formal immersion from $TM$ in $TV$ is a necessary condition for the existence of an immersion of $M$ in $V$. 

\begin{definition}[Homotopy between formal immersions]
	 Two formal immersions, $F$ and $G$ are called formally homotopic (or just homotopic) if there is a homotopy $H_t: TM \rightarrow TV$ of formal immersions such that $H_0 = F$ and $H_1 = G$.
 \end{definition}

Two immersions $f$ and $g$ are called formally homotopic if $Df$ and $Dg$ are homotopic as formal immersions. To be precise there exists a formal homotopy $H_t$ covering a smooth homotopy $f_t$ joining $f$ and $g$ such that $H_0 = Df$ and $H_1 = Dg$. Assume that $dim(V) \geq dim(M) + 1$. Let $Imm(M,V)$ denote the set of all immersions of $M$ in $V$ and let $Mono(TM,TV)$ denote the set of all formal immersions. Let $I : Imm(M,V) \rightarrow Mono(TM,TV)$ be the inclusion map given by the tangent bundle monomorphism induced by an immersion. 

\begin{theorem}[The Smale-Hirsch h-principle for immersion](\cite{Hi})
	  The map $I$ is a homotopy equivalence.
\end{theorem}

So, $I$ induces set bijection from $\pi_0(Imm(M,V))$ to $\pi_0(Mon(TM,TV))$. This implies that the existence of a formal immersion is also sufficient for the existence of an immersion. Moreover, the isomorphism that $I$ induces from $\pi_1(Imm(M,V))$ to $\pi_1(Mon(TM,TV))$ implies that if $f_0$ and $f_1$ are two immersions which are formally homotopic, then they are regularly homotopic. 

    \


 
We now discuss the obstruction theoretic problem for the existence of a formal immersion and the classification of formal immersions. For more details on immersion theory we refer to \cite{Hi}. For related terminologies and notions from fiber bundle and obstruction theory and we refer to \cite{St}.        \\

\textbf{Obstructions to formal immersion and homotopy:} Given a manifold $N$, $T_kN$ will denote the bundle of $k$-frames associated to $TN$. A formal immersion of $M^n$ into $\mathbb{R}^N$ defines an $SO(n)$--equivariant map from $T_nM$ to $V_{N,n}$. Here, $V_{N,n}$ denotes the real Stiefel manifold consisting of all oriented $n$-frames in $\mathbb{R}^N$. According to \cite{Hi}, the homotopy classes of the formal immersions are in one-one correspondence with the homotopy classes of such $SO(n)$--equivariant maps to $V_{N,n}$, i.e., with homotopy classes of cross sections of the associated bundle of $T_nM$ with fiber $V_{N,n}$. Thus, the problem is reduced to looking at the obstructions to the existence of a section $s$ of this associated bundle. Such obstructions lie in the groups $H^i(M^n;\pi_{i-1}(V_{N,n}))$, for $1 \leq i \leq n$.

Moreover, two formal immersions $F$ and $G$ are homotopic if and only if their corresponding sections $s_F$ and $s_G$ to the associated $V_{N,n}$--bundle are homotopic. Thus, the homotopy obstructions between two formal immersions $F$ and $G$ lie in $H^i(M^n;\pi_i(V_{N,n}))$ for $1 \leq i \leq n$.  

\subsection{Almost contact structure}
 We recall the notion of almost contact structure. 
 
\begin{definition}
 Consider a real vector bundle $p: E \rightarrow B$ of rank $2m$. An almost complex structure on $E$ is a smooth assignment of automorphisms $J_p : F_p \rightarrow F_p$ for all point $p$ in $B$ such that $J_p^2=-Id$. 	
\end{definition}

The set of all complex structures on $\mathbb{R}^{2m}$ is homeomorphic to $\Gamma_m=SO(2m)/U(m)$ (\cite{Ge}, Lemma $8.1.7$). Thus, the existence of an almost complex structure on $E$ is equivalent to the existence of a section of the associated $\Gamma_m$-bundle. 

\begin{definition} An almost contact structure on an odd dimensional manifold $N^{2n+1}$ is an almost complex structure on its stable tangent bundle $TN\oplus\varepsilon^1_N$.
\end{definition}

Thus, an almost contact structure on $N$ is an almost complex structure on $N \times \mathbb{R}$. So every almost contact structure on $N$ is given by a section of the associated $\Gamma_{n+1}$--bundle of $T(N\times\mathbb{R})$. 


\begin{definition} Two almost contact structures are said to be in the same homotopy class, if their corresponding sections to the associated $\Gamma_{n+1}$--bundle are homotopic. 
\end{definition}

The existence of an almost contact structure on $N$ is a necessary condition for the existence of a contact structure. For open manifolds, Gromov (\cite{Gr}) proved the following h-principle showing that this condition is also sufficient. 

\begin{theorem} (Gromov, \cite{Gr}) \label{existence of contact structure gromov}
	Let $K$ be a sub-complex of an open manifold $V$. Let $\bar{\xi}$ be an almost contact structure on $V$ which restricts to a contact structure in a neighborhood $Op(K)$ of $K$. Then one can homotope $\bar{\xi}$, relative to $Op(K)$, to a contact structure $\xi$ on $V$.    
\end{theorem} 

For closed manifolds, the corresponding h-principle follows from the work of Borman, Eliashberg and Murphy (\cite{BEM}). In particular, they showed that in every homotopy class of an almost contact structure there is at least one contact structure called $overtwisted$ (see \cite{BEM}). Two such overtwisted contact structures are isotopic if and only if they are homotopic as almost contact structures. \cite{BEM} gives a parametric version of Theorem \ref{existence of contact structure gromov} that holds for both open and closed contact manifolds.

\begin{theorem}[Borman, Eliashberg and Murphy, \cite{BEM}] \label{homotopy of contact structure bem}
	Let $K \subset M^{2n+1}$ be a closed subset. Let $\xi_1$ and $\xi_2$ be two overtwisted contact structures on $M$ that agree on some $Op(K)$. If $\xi_1$ and $\xi_2$ are homotopic as almost contact structures over $M \setminus K$, then $\xi_1$ and and $\xi_2$ are homotopic as contact structures relative to $Op(K)$.  
\end{theorem}

So, by Gray's stability, $\xi_1$ and $\xi_2$ are isotopic contact structures.

\

 The obstructions to the existence of an almost contact structure on $N^{2n+1}$ lie in the groups $H^i(N;\pi_{i-1}(\Gamma_{n+1}))$, for $1 \leq i \leq 2n+1$. The homotopy obstructions between two almost contact structures on $N$ lie in the groups $H^i(N;\pi_i(\Gamma_{n+1}))$, for $1 \leq i \leq 2n+1$. The stable homotopy groups of $\Gamma_{n}$ were computed by Bott (\cite{B}). 

\begin{theorem}[Bott, \cite{B}] \label{bott's theorem}
	For $q \leq 2n-2$
	
	\[
	 \pi_q(\Gamma_{n}) = \pi_{q+1}(SO) = \begin{cases}
	
	0 \ \ & \text{for} \ \ q \equiv 1,3,4,5 \ (\text{mod} \ 8) \\
	\mathbb{Z} \ \ & \text{for} \ \ q \equiv 2,6 \ (\text{mod} \ 8) \\ \mathbb{Z}_2 \ \ & \text{for} \ \ q \equiv 0,7 \ (\text{mod} \ 8)
	\end{cases}
	\]
\end{theorem}

 Next, we define the notion of a \emph{null-homotopic} almost contact structure on $\pi$--manifold. 

\begin{definition}
	When $M^{2n+1}$ is a $\pi$--manifold, $TM \oplus \varepsilon^1_M \cong \varepsilon^{2n+2}_M$ always admits a trivial almost complex structure that assigns over each point of $M$ the standard complex structure $J_0(n+1)$ on $\mathbb{C}^{n+1}$. The corresponding homotopy class of almost contact structures on $M$ is called null-homotopic.
\end{definition}

\subsection{Obstructions to contact immersion} 
Similar to the notion of formal immersion one can define a \textit{formal contact immersion} or a \textit{contact monomorphism}.

\begin{definition}
	A formal immersion $F: TM \rightarrow TV$ of $(M,\xi)$ in $(V,\eta)$ is called a contact monomorphism if $F(\xi) = F(TM)\cap{\eta}$.
\end{definition}
	
Gromov (\cite{Gr}) proved the following h-principle for contact immersions.

\begin{theorem}[Gromov] \label{h-principle for contact immersion}
	Let $(M,\xi)$ and $(V,\eta)$ be contact manifolds of dimensions $2n+1$ and $2N+1$ respectively. Assume that $n \leq N-1$. A contact monomorphism $F_0: TM \rightarrow TV$, covering an immersion $f_0: M \rightarrow V$, is formally homotopic to $F_1 = df_1$ for some contact immersion $f_1: M \rightarrow V$.
\end{theorem}

For contact embedding, Gromov (\cite{Gr}) proved the following theorem. The statement here is taken from \cite{EM}.

\begin{theorem}[Gromov] \label{h-principle for contact embedding}
	Let $(M,\xi)$ and $(V,\eta)$ be contact manifolds of dimension $2n+1$ and $2N+1$ respectively. Suppose that the differential $F_0 = Df_0$ of an embedding $f_0 : (M,\xi) \rightarrow (V,\eta)$ is homotopic via a homotopy of monomorphisms $F_t : TM \rightarrow TV$ covering $f_0$ to a contact monomorphism $F_1 : TM \rightarrow TV$.
	
	\begin{enumerate}
		\item 	Open case: If $n \leq N-1$ and the manifold $M$ is open, then there exists an isotopy $f_t : M \rightarrow V$ such that the embedding $f_1 : M \rightarrow W$ is contact and the differential $Df_1$ is homotopic to $F_1$ through contact monomorphisms.
		\item 	Closed case: If $n \leq N-2$, then the above isotopy $f_t$ exists even if $M$ is closed.
	\end{enumerate}
	
\end{theorem}

Thus, if an embedding is regularly homotopic to a contact immersion, then it is isotopic to a contact embedding. We now discuss the analogous obstruction problem for the existence of a contact monomorphism.

\

\textbf{Obstructions to contact monomorphism :}  Consider a symplectic vector space $(X,\omega)$. Let $J$ be an $\omega$-compatible almost complex structure (i.e., $\omega(Ju,Jv) = \omega(u,v)$ and $\omega(u,Ju)>0$ for all $u,v \in X\setminus\{0\}$). If $Y$ is a symplectic subspace of $(X,\omega)$, then $Y$ has to be a $J$-subspace of $(X,J)$ and vice-versa. An almost complex structure $J_\xi$ on the contact hyperplane bundle $\xi = Ker \{\alpha\}$ is called $\xi$--compatible, if it is compatible with the conformal symplectic structure on $\xi$ induced by $d\alpha$. A contact monomorphism takes $\xi$ to a symplectic sub-bundle of $\eta$. If $J_\eta$ is an $\eta$--compatible almost complex structure, then the contact monomorphism takes $\xi$ to a $J_{\eta}$-sub-bundle of $(\eta,J_{\eta})$. So, finding a contact monomorphism from $(TM,\xi)$ to $(T\mathbb{R}^{2N+1},\eta)$ is equivalent to finding a $U(n)$--equivariant map from the complex $n$-frame bundle associated to $\xi$ to $V^\mathbb{C}_{N,n}$. In other words, finding a contact monomorphism is equivalent to the existence of a section of the associated $V^\mathbb{C}_{N,n}$--bundle of $TM$ (see $2.2$ in \cite{Ka}). Here, $V^\mathbb{C}_{N,n}$ denotes the complex Stiefel manifold. Thus, $(M,\xi)$ has a contact monomorphism in $(V,\eta)$ if and only if all the obstructions classes in $H^{i}(M;\pi_{i-1}(V^\mathbb{C}_{N,n}))$ vanish for $1\leq{i}\leq2n+1$.    \\

\textbf{From contact immersion to contact embedding :}  In the previous sections, we saw that any formal immersion of $M^{2n+1}$ in $\mathbb{R}^{2N+1}$ is given by a section $s_F$ of the associated $V_{2N+1,2n+1}$-bundle of $TM$. For a contact monomorphism $F_\mathbb{C}$, let $s_{F_\mathbb{C}}$ denote the corresponding section to the associated $V^\mathbb{C}_n(\eta)$--bundle. $s_{F_\mathbb{C}}$ also induces a section map $s_\mathbb{C}$ to the associated $V_{2N+1,2n+1}$-bundle via the inclusion $V^\mathbb{C}_{N,n} \subset V_{2N,2n} \subset V_{2N+1,2n+1}$. Let $\iota: M \hookrightarrow V$ be an embedding and let $s_\iota$ denote the corresponding section to the associated $V_{2N+1,2n+1}$-bundle. The homotopy obstructions between $s_\iota$ and $s_\mathbb{C}$ lie in the groups $H^{i}(M;\pi_{i}(V_{2N+1,2n+1}))$, for $1 \leq i \leq2n+1$. If all of these obstructions vanish, then by Theorem \ref{h-principle for contact embedding}, $\iota$ can be isotoped to a contact embedding. \\

For example, let $(V,\eta)=(\mathbb{R}^{4n+3},\xi_{std})$. Note that $V_{4n+3,2n+1}$ is $(2n+1)$-connected and $V^{\mathbb{C}}_{2n+1,n}$ is $(2n+2)$-connected. So, all of the groups $H^{i}(M;\pi_{i-1}(V^{\mathbb{C}}_{2n+1,n})$ and $H^{i}(M;\pi_{i}(V_{2N+1,2n+1})$ vanish, for $1\leq i \leq2n+1$. By the Whitney embedding theorem, any smooth $(2n+1)$-manifold embeds into ${\mathbb{R}}^{4n+3}$. Thus, we get the result of Gromov (\cite{Gr}) saying that every contact manifold $(M^{2n+1},\xi)$ has an isocontact embedding in $({\mathbb{R}}^{4n+3},\xi_{std})$.

\begin{remark} \label{remark on chern class and normal bundle}
	When the embedding co-dimension is $\leq dim(M)-1$, there is a natural topological obstruction to contact embedding. It has the following description. If $\iota:(M^{2n+1},\xi) \hookrightarrow (\mathbb{R}^{2N+1},\eta)$ is a contact embedding, then the normal bundle $\nu(\iota) = \iota^*(\eta)/{\xi}$ has an induced complex structure on it. So we have the following relation of total Chern classes. $$ c(\xi \oplus \nu(\iota)) = \iota^*(\eta) = 1 $$. Let $\bar{c}_j(\xi)$ denote the $j^{th}$ order cohomology class in $(1+c_1(\xi)+c_2(\xi)+...+c_n(\xi))^{-1}$. Since the Euler class of the normal bundle of an embedding in $\mathbb{R}^{2N+1}$ is zero, $$c_{N-n}(\nu(\iota)) = 0 \Leftrightarrow \bar{c}_{N-n}(\xi) = 0$$ This gives a condition on the Chern classes of $\xi$. Thus, for isocontact embedding of co-dimension $\leq dim(M)-1$, one has to restrict the problem on the contact structures whose Chern classes satisfy this condition. For isocontact embedding with trivial symplectic normal bundle, the following holds. $$ \xi \oplus \nu(\iota) \cong \xi \oplus \varepsilon_M^{N-n}(\mathbb{C}) = \eta|_{\iota(M)} \cong \varepsilon_M^{N}(\mathbb{C})$$
	Thus, $c_i(\xi \oplus \varepsilon_M^{N-n}(\mathbb{C})) = 0 \Leftrightarrow c_i(\xi)=0 $, for $1 \leq i \leq n$. Therefore, $\bar{c}_{N-n}(\xi) = 0$.
\end{remark}

\subsection{Smooth embeddings in Euclidean space}

The following generalization of Whitney embedding theorem is due to Haefliger and Hirsch (\cite{HH}).

\begin{theorem}[Haefliger-Hirsch] \label{haefligar-hirsch} If $M^n$ is a closed orientable $k$-connected $n$-manifold $(0 \leq k \le \frac{1}{2}(n-4))$, then $M^n$ embeds in ${\mathbb{R}}^{2n-k-1}$. 
	
\end{theorem}

If one further assumes that $M^n$ is a $\pi$--manifold, then we have the following result due to Sapio (\cite{Sa}).

\begin{theorem} [Sapio] \label{thm:sapio} Let $M^n$ be a $k$-connected, $n$-dimensional $\pi$--manifold $(n \geq 5$, and $k \leq [n/2])$. Assume that $n \not\equiv 6 \ (mod \ 8)$. Then
	\begin{enumerate}
		\item $M^n$ almost embeds in ${\mathbb{R}}^{2n-2k-1}$ with a trivial normal bundle.
		\item If $M^n$ bounds a $\pi$--manifold, then $M^n$ embeds in ${\mathbb{R}}^{2n-2k-1}$ with a trivial normal bundle.
	\end{enumerate}
\end{theorem}

\section{contact embedding via h-principle}

 Note that $TN^n \bigoplus \epsilon^{k+1}_N \cong \epsilon^{n+k+1}_N \Leftrightarrow TN^n \bigoplus \epsilon^1_N = \epsilon^{n+1}_N$ (see corollary $1.4$, p-$70$ of \cite{Kos}). Therefore, $M^{2n+1}$ is a $\pi$-manifold if and only if $M^{2n+1}$ embeds in the Euclidean space $\mathbb{R}^{d}$ with a trivial normal bundle, for some $d \geq 2n+2$. The following lemma is the main ingredient to prove Theorem \ref{1st theorem} .

\begin{lemma}
	If an almost contact manifold $M^{2n+1}$ embeds in ${\mathbb{R}}^{2N+1}$ with a trivial normal bundle, then there exists a contact structure $\xi_0$ such that $(M,\xi_0)$ isocontact embeds into $({\mathbb{R}}^{2N+3},\xi_{std})$ ($N-n\geq1$). 
\end{lemma}

\begin{proof}  
	By assumption, there is an embedding $\iota : M \hookrightarrow {\mathbb{R}}^{2N+1}$ with normal bundle of embedding $\nu(\iota)$ trivial. Since M is a $\pi$--manifold, $TM\oplus\epsilon^1_M \cong \epsilon^{2n+2}_M$. So, any section to the associated $\Gamma_{n+1}$-bundle of $TM\oplus\epsilon^1_M$ is given by a homotopy class of map $s: M \rightarrow \Gamma_{n+1}$. By \cite{BEM}, in every homotopy class of an almost contact structure there is a genuine contact structure. Fix a homotopy class of an almost contact structure on $M^{2n+1}$. Let $\xi$ be a contact structure representing it. Let $E(\nu)$ denote the total space of $\nu(\iota)$. Since $\xi$ is co-orientable, $$TE(\nu) \oplus \varepsilon^{1} \cong TM \oplus \nu(\iota) \oplus \varepsilon^1 \cong \xi \oplus \varepsilon^2 \oplus \varepsilon^{2(N-n)}$$. We now define a contact structure on the tubular neighborhood $E(\nu)$ of $M$, such that its restriction to $M$ is contact. Let $\alpha$ be a contact form representing $\xi$. Let $(r_1,\theta_1,r_2,\theta_2,...,r_{N-n},\theta_{N-n})$ be a cylindrical co-ordinate system on $\mathbb{D}^{2(N-n)}$. The 1-form $\tilde{\alpha} = \alpha + \Sigma^{N-n}_{i=1}(r_i^2d\theta_i)$ defines a contact structure on $E(\nu) \cong M\times\mathbb{D}^{2N-2n}$ that restricts to the contact structure $\xi$ on $M$. Let $J_{\xi}$ be an almost complex structure on the stable tangent bundle of $TM$ that induces the contact structure $\xi$ on $M$. Put the standard complex structure $J_0(N-n)$ on the normal bundle $\nu(\iota)$ and define an almost complex structure on $TE(\nu(\iota)) \oplus \epsilon^1_E$ given by $J_{\xi} \oplus J_0(N-n)$. Note that over each fiber of $E(\nu)$, $d\tilde{\alpha}$ restricts to the standard symplectic structure on $\mathbb{D}^{2N-2n}$ compatible with $J_0(N-n)$. So, the almost contact structure associated to $\tilde{\alpha}$ is the same as the almost contact structure induced by $J_{\xi} \oplus J_0(N-n)$. If we can extend this almost contact structure on $E(\nu)$ to all of $\mathbb{R}^{2N+1}$, then by Theorem \ref{existence of contact structure gromov}, we will get a contact embedding of $(M,\xi)$ into $(\mathbb{R}^{2N+1},\eta_0)$ for some contact structure $\eta_0$.
	  
	Now we show how to extend the section $s_{\xi}: M \rightarrow \Gamma_{N+1}$, given by $J_{\xi} \oplus J_0(N-n)$, to all of $\mathbb{R}^{2N+1}$. The obstructions to such an extension lie in $H^{i+1}(\mathbb{R}^{2N+1},M;\pi_i(\Gamma_{N+1})) \cong H^i(M;\pi_i(\Gamma_{N+1}))$, for $1 \leq i \leq 2n+1$. Consider the section $s_{\eta} : M \rightarrow \Gamma_{N+1}$ induced by $\eta|_{\iota(M)}$, for some contact structure $\eta$ on $\mathbb{R}^{2N+1}$. If $s_{\xi}$ is homotopic to $s_{\eta}$, then the obstructions vanish and $s_{\xi}$ extends to all of $\mathbb{R}^{2N+1}$.
	 
	Let $In: \Gamma_{n+1} \hookrightarrow \Gamma_{N+1}$ be the inclusion map given by $J(n+1) \longmapsto J(n+1) \oplus J_0(N-n)$. Consider the fibration $\Gamma_m \xrightarrow{j_m} \Gamma_{m+1} \rightarrow S^{2m}$ \cite{Ha}. Here, $j_m$ denotes the inclusion map that sends $J(m)$ to $J(m) \oplus J_0(2)$. From this we get the following long exact sequence.
	 
	$$ ...\longrightarrow \pi_{i+1}(S^{2m}) \longrightarrow \pi_{i}(\Gamma_m) \longrightarrow \pi_{i}(\Gamma_{m+1}) \longrightarrow \pi_{i}(S^{2m}) \longrightarrow...$$
	
	It follows that $j_m$ induces isomorphism on $\pi_i$ for $i\leq 2m-2$ and onto homomorphism for $i = 2m-1$. Note that the composition map $\Gamma_{n+1} \xrightarrow{j_N \circ j_{N-1} \circ ... \circ j_{n+1}} \Gamma_{N+1}$ is the same as the one defined by $In$. Here, $j_{n+1}$ induces isomorphism on $\pi_{i}$ for $i\leq 2n$ and onto homomorphism for $i=2n+1$. For $l \geq n+2$, $j_l$'s induce isomorphisms on $\pi_i$'s, for all $i \leq 2n+1$. Therefore, $In$ induces isomorphism on the $i^{th}$-homotopy group, for $1 \leq i \leq 2n$ and onto homomorphism for $i = 2n+1$. So, we can choose the homotopy class of $\xi$ so that the homotopy class of the image of the corresponding section $s_{\xi} : M \rightarrow \Gamma_{n+1}$ under $In$ is the same as the homotopy class of the map $s_{\eta} : M \rightarrow \Gamma_{N+1}$. Thus, there is a contact structure $\xi_0$ on $M$ such that the corresponding section $s_{\xi_0}: M \rightarrow \Gamma_{N+1}$ extends to all of $\mathbb{R}^{2N+1}$. Therefore, $(M,\xi_0)$ contact embeds in $(\mathbb{R}^{2N+1},\eta_0)$ for some contact structure $\eta_0$ on $\mathbb{R}^{2N+1}$. By Theorem \ref{h-principle for contact embedding}, $(\mathbb{R}^{2N+1},\eta_0)$ contact embeds in $(\mathbb{R}^{2N+3},\xi_{std})$. Hence, $(M,\xi_0)$ isocontact embeds in $(\mathbb{R}^{2N+3},\xi_{std})$. \end{proof} 

	\begin{remark}
	For $n \geq 4$, the groups $\pi_{2n+1}(\Gamma_{n+1})$ have the following values \cite{Ha}:

	\begin{equation}
	\pi_{2n+1}(\Gamma_{n+1}) = \begin{cases}
	
	\mathbb{Z}\oplus\mathbb{Z}_2 \ \ & \text{for} \  \ n \equiv 3 \ (\text{mod} \ 4) \\
	\mathbb{Z}_{(n-1)!} \ \ & \text{for} \  \ n \equiv 0 \ (\text{mod} \ 4) \\ \mathbb{Z} \ \ & \text{for} \  \ n \equiv 1 \ (\text{mod} \ 4) \\
	\mathbb{Z}_{\frac{(n-1)!}{2}} \ \ & \text{for} \  \ n \equiv 2 \ (\text{mod} \ 4)
	\end{cases}
	\end{equation}
	
	Whereas, $\pi_{2n+1}(\Gamma_{N+1})$ is either $0$ or $\Z_2$. So, for $n \geq 4$, the onto map induced by $In$ on $\pi_{2n+1}$ has a non-trivial kernel. Thus, we can actually choose a homotopy class of almost contact structures on $M$ which is not null-homotopic and which isocontact embeds in $(\mathbb{R}^{2N+3},\xi_{std})$.	
	\end{remark}

\begin{proof}[\textbf{Proof of Theorem \ref{1st theorem}}] Since every homotopy sphere $\Sigma^{2n+1}$ admits an almost contact structure, $M^{2n+1}\sharp \Sigma^{2n+1}$ also admits an almost contact structure. By statement $2$ of the Theorem \ref{thm:sapio}, if $M^{2n+1}$ bounds a $\pi$--manifold then it satisfies the hypothesis of Lemma $3.1$, for $N= 2n-k$. By statement $1$ of the Theorem \ref{thm:sapio}, whenever $M^{2n+1}$ does not bound a $\pi$--manifold, we can take connected sum with a suitable homotopy sphere and embed the resulting $\pi$--manifold in $\mathbb{R}^{4n-2k+1}$ with a trivial normal bundle. This hypothesis on the normal bundle of embedding is the only thing that we need to prove the present theorem. Therefore, it is enough to prove statements $(1)$ to $(3)$, for manifolds that bound $\pi$--manifold.

\textbf{(1)} The result follows from Lemma $3.1$.

\textbf{(2)} As discussed in the proof of Lemma $3.1$, given any contact structure $\xi$ on $M^{2n+1}$ and an embedding $\iota: M \rightarrow \mathbb{R}^{4n-2k+1}$ with a trivial normal bundle $\nu$, the obstructions to extend the almost contact structure on $E(\nu)$ to all of ${\mathbb{R}}^{4n-2k+1}$, lie in the groups $H^i(M;\pi_i(\Gamma_{2n-k+1}))$, for $ 1\leq i \leq2n+1$. Since $M$ is $k$-connected, there are no obstructions in dimensions $1$ to $k$. Being $k$-connected also implies that $M\smallsetminus\{pt.\}$ deformation retracts onto the $(2n-k)$-skeleton of $M$. So, the obstructions in dimensions $(2n-k+1)$ to $2n$ also vanish. Thus, we now only consider values of $i$ in $\{k+1, k+2,...,2n-k+1\}$. Since $k \leq n-1$, by the theorem of Bott (\ref{bott's theorem}), we get the following for $1\leq i \leq 2n+1$.
	
	  $$\pi_{i}(\Gamma_{2n-k+1}) = \begin{cases}
	0  \  \   & for \ \ i \equiv 1,3,4,5 \ (mod \ 8) \\ {\mathbb{Z}}_2   \  \ & for \ \ i \equiv 0,7 \ (mod \ 8) \\ \mathbb{Z} \  \ & for \ \ i \equiv 2,6 \ (mod \ 8)
	\end{cases} $$
	
	So, for $i \equiv 1,3,4,5 \ (mod \ 8)$, there are no obstructions. For $i \equiv 0,2,6,7 \ (mod \ 8)$, $H^i(M^{2n+1},\tilde{G}) \cong H_{2n-i+1}(M^{2n+1},\tilde{G}) = 0$ by hypothesis. Here, $\tilde{G}$ is either $\mathbb{Z}$ or $\mathbb{Z}_2$. Moreover, for $n \not \equiv 3 \ (mod \ 4)$, $\pi_{2n+1}(\Gamma_{2n-k+1}) = 0$. Hence, there is no obstructions in the top dimension. Thus, we can extend the almost contact structure for any $\xi$ and the result follows.
	
\textbf{(3)} Note that every assumption in statement $(2)$ holds for statement $(3)$, except that now we have $H^i(M^{2n+1},\tilde{G}) \cong H_{2n-i+1}(M^{2n+1},\tilde{G}) = 0$ for $i \equiv 0,7 \ (mod \ 8)$. Thus, we are left to show that the obstructions vanish for $i \equiv 2,6 \ (mod \ 8)$. We now claim that in the proof of Lemma $3.1$, both $s_{\xi}$ and $s_{\eta}$ factors through the map $\tilde{j}$ in the fibration $$U(N+1) \rightarrow SO(2N+2) \xrightarrow{\tilde{j}_{N+1}} \Gamma_{N+1}$$. Since $\eta$ was a contact structure on $\mathbb{R}^{2N+1}$, the assertion is clear for $s_{\eta}$. The reason for $s_{\xi}$ is the following. Since $\xi$ is an $SO$-contact structure, the almost contact structure associated to $\xi$, $M \rightarrow \Gamma_{n+1}$, factors through the map $SO(2n+2) \xrightarrow{\tilde{j}_{n+1}} \Gamma_{n+1}$. Let $\hat{i} : SO(2n+2) \rightarrow SO(2N+2)$ denote the inclusion map given by $A \mapsto A \oplus I_{2(N-n)}$. Recall that $\tilde{j}$ takes a matrix $A \in SO(2m+2)$ to $A^{-1}J_0(m+1)A \in \Gamma_{m+1}$. The assertion then follows from the commutative diagram below.
	
	 \begin{equation*}
	\begin{CD}
	SO(2n+2)   @>\hat{i}>>  SO(2N+2)\\
	@VV\tilde{j}_{n+1}V        @VV\tilde{j}_{N+1}V\\
	\Gamma_{n+1}  @>In>> \Gamma_{N+1}
	\end{CD}
	\end{equation*}
	
Thus, the homotopy obstructions come from the groups $H^i(M;\pi_i(SO(2N+2)))$. Since $\pi_{i}(SO) = 0$ for $i \equiv 2,6 \ (mod \ 8)$, the obstructions vanish. Therefore, following the proof of Lemma $3.1$, we get an isocontact embedding of $(M,\xi)$ into $(\mathbb{R}^{4n-2k+3},\xi_{std})$.

\end{proof}

\begin{proof}[\textbf{Proof of Corollary \ref{1st corollary}}]  
	
	\textbf{(1)} Follows from Theorem \ref{1st theorem} by putting $k = n-1$.
	
	\textbf{(2)} Following the proof of Theorem \ref{1st theorem}, we can see that the only obstructions to extending the almost contact structure on the normal bundle to all of ${\mathbb{R}}^{2n+3}$ lie in the groups $H^n(M;\pi_n(SO))$ and $H^{n+1}(M;\pi_{n+1}(SO))$. Since both $\pi_n(SO)$ and $\pi_{n+1}(SO)$ vanish for $n \equiv 4,5 \ (mod \ 8)$, the result follows.

\end{proof}

\begin{proof}[\textbf{Proof of Theorem \ref{2nd theorem}}] 

\textbf{(1)}  The existence of a contact monomorphism from $(TM^{2n+1},\xi)$ to $({T\mathbb{R}}^{2N+1},\eta_{st})$ is equivalent to the existence of a section $s: M \rightarrow V_{N,n}^{\C}$ of the associated bundle of $TM$. Since $TM$ is trivial, such a section always exists. Since any parallelizable manifold $M^{2n+1}$ immerses in $\mathbb{R}^{2n+3}$ and has a contact monomorphism in $(T\mathbb{R}^{2n+3},\xi_{std})$, by Theorem \ref{h-principle for contact immersion}, $(M,\xi)$ isocontact immerses in $(\mathbb{R}^{2n+3},\xi_{std})$.

\textbf{(2)} Any section corresponding to a contact monomorphism also induces a section $s_0$ to the associated $V_{2N+1,2n+1}$-bundle of $TM$. Assume that $M$ is $(2k-1)$-connected. By Theorem \ref{haefligar-hirsch}, there exists an embedding $f: M^{2n+1} \rightarrow \mathbb{R}^{4n-2k+3}$. Let $s_f$ be the corresponding section to $V_{4n-2k+3,2n+1}$. The homotopy obstructions between $s_0$ and $s_f$ lie in the groups $H^i(M^{2n+1},\pi_i(V_{4n-2k+3,2n+1}))$, for $1\leq i\leq 2n+1$. Since $M^{2n+1}\smallsetminus{\mathbb{D}}^{2n+1}$ deformation retracts onto the $(2n-2k+1)$-skeleton of $M$ and $V_{4n-2k+3,2n+1}$ is $(2n-2k+1)$-connected, there are no obstructions till dimension $2n$. Therefore, the only homotopy obstruction lies in $H^{2n+1}(M^{2n+1},\pi_{2n+1}(V_{4n-2k+3,2n+1}))$. By \cite{HM}, for $k = 3$, and $n \equiv 0,1 \ (mod \ 4)$, $\pi_{2n+1}(V_{4n-3,2n+1}) = 0$ . Thus, by Theorem \ref{h-principle for contact embedding}, $(M^{2n+1},\xi)$ has an isocontact embedding in $(\mathbb{R}^{4n-3},\xi_{std})$.

\end{proof}

Note that the proof of statement $(2)$ in Theorem \ref{2nd theorem} does not necessarily require a parallelizable manifold. In general, the following can be said.

\begin{proposition}
	A $5$-connected contact manifold $(M^{2n+1},\xi)$ admits an isocontact embedding in $(\mathbb{R}^{4n-3},\xi_{std})$ for $n \geq 3$ and $n \equiv 0,1 \ (mod \ 4)$, if it admits an isocontact immersion in $(\mathbb{R}^{4n-3},\xi_{std})$. 
\end{proposition}

\begin{proof}[\textbf{Proof of Corollary \ref{2nd corollary}}]
	Consider $\mathbb{R}^{N}$ as $\mathbb{R}^{N-2} \times \mathbb{R}^2$. It is well known that $\mathbb{R}^N \setminus (\mathbb{R}^{N-2} \times \{0\})$ can be decomposed as $\mathbb{R}^{N-1} \times S^1$. This is the so called standard open book decomposition of $\mathbb{R}^{N}$ (see \cite{Ge} or \cite{E}). Say, $M$ is embedded in $\mathbb{R}^{N-1}$. Then using this open book description we can see that $M \times S^1$ naturally embeds in $\mathbb{R}^{N-1} \times S^1 \subset \mathbb{R}^N$. Starting with an embedding of $N^{2n-1}$ in $\mathbb{R}^{2N+1}$ with trivial normal bundle, we can then apply this procedure twice to get an embedding of $N^{2n-1} \times (S^1 \times S^1)$ in $\mathbb{R}^{2N+3}$ with trivial normal bundle. The result then follows from Lemma $3.1$.
	
\end{proof}

	\subsection{Contact embedding of co-dimension $\geq 4$} \label{codim 4 embedn} For embeddings of co-dimension $\geq 4$, we can actually get isocontact embedding in the standard contact structure. Let $\iota$ be an isocontact embedding of $(M,\xi)$ in $(\mathbb{R}^{2N+1},\eta)$. Any two contact structures on $\mathbb{R}^{2N+1}$ are homotopic as almost contact structures. Let $H_t : T\mathbb{R}^{2N+1} \rightarrow T\mathbb{R}^{2N+1}$ be a formal homotopy covering the identity map of $\mathbb{R}^{2N+1}$ such that $H_0 = Id$, $H_1(\eta) = \xi_{std}$ and $H_t(\eta)$ is an almost contact structure on $\mathbb{R}^{2N+1}$ for all $t \in (0,1)$. Then $H_t \cdot D\iota$ gives a formal homotopy covering $\iota$ and $H_1 \cdot D\iota$ is a contact monomorphism of $(TM,\xi)$ into $(T\mathbb{R}^{2N+1},\xi_{std})$. If we assume that $N-n \geq 2$, then by Theorem \ref{h-principle for contact embedding}, $(M^{2n+1},\xi)$ isocontact embeds in $(\mathbb{R}^{2N+1},\xi_{std})$. Thus, for embedding of co-dimension $\geq 4$, we also get contact embedding in the standard contact structure. Moreover, this shows that for $n-k \geq 2$, we can improve the dimension of embedding in Theorem \ref{1st theorem} from $4n - 2k + 3$ to $4n - 2k + 1$. Using this fact with Lemma $3.1$, one can find interesting examples of contact embedding of non-simply connected manifolds. For example, by \cite{MR}, the $7$-dimensional real projective space $RP^7$ embeds in $\mathbb{R}^{13}$ with trivial normal bundle. Thus, $RP^7$ contact embeds in $(\mathbb{R}^{13},\xi_{std})$. Let us look at another simple class of such examples. It is a well known theorem of Hirsch that every oriented, closed $3$-manifold $M^3$ embeds in $\mathbb{R}^5$. By \cite{CS}, every closed, oriented $4$-manifold $V^4$, whose second Stiefel-Whitney class and signature vanish, embeds in $\mathbb{R}^6$. Since, the Euler class of the normal bundle of an embedding vanishes, each of these embeddings has trivial normal bundle. Thus, $W^7 = M^3 \times V^4$ embeds in $\mathbb{R}^{11}$ with trivial normal bundle. Hence, $W^7$ contact embeds in $(\mathbb{R}^{11},\xi_{std})$.              

\

We now prove Theorem \ref{3rd theorem}. The idea of the proof is essentially contained in \cite{EF} (proof of Theorem $1.25$). 

\begin{proof}[\textbf{Proof of Theorem \ref{3rd theorem}}].  
	
	 Since $\iota_1$ and $\iota_2$ are contact embeddings with trivial normal bundle, there exists a contact form $\alpha$ representing $\xi$ and suitable  neighborhoods $\mathcal{N}_j$ of $\iota_j(M)$ in $(\mathbb{R}^{2N+1},\eta_{ot})$, for $j=1,2$, which are contactomorphic to $(M \times \mathbb{D}^{2(N-n)}, \alpha + \Sigma^{N-n}_{i=1}r_i^2d\theta_i)$. Now, $\iota_1$ and $\iota_2$ are isotopic and have isomorphic symplectic normal bundle. Thus, one can use the contact tubular neighborhood theorem for contact submanifold (Theorem $2.5.15$, \cite{Ge}) to get an ambient isotopy $\Phi_t : \mathbb{R}^{2N+1} \longrightarrow \mathbb{R}^{2N+1}$ such that $\Phi_1$ restricts to a contactomorphism from $(\mathcal{N}_1,\eta_{ot}|_{\mathcal{N}_1})$ to $(\mathcal{N}_2,\eta_{ot}|_{\mathcal{N}_2})$ and $\Phi_1(\iota_1(M)) = \iota_2(M)$. Moreover, since we can assume that the isotopy between $\iota_1$ and $\iota_2$ is supported in the complement of an overtwisted contact ball, $\Phi_t$ can be chosen so that it restricts to the identity map on that overtwisted contact ball in the complement of $\mathcal{N}_1$ in $(\mathbb{R}^{2N+1},\eta_{ot})$. Thus, the distribution $(\Phi_1)_*\eta_{ot}$ induces a contact structure on $\mathbb{R}^{2N+1}$ that is overtwisted in the complement of $\mathcal{N}_2$. Now, we look at the homotopy obstructions between $(\Phi_1)_*\eta_{ot}$ and $\eta_{ot}$ relative to $\mathcal{N}_2$. All such obstructions lie in $H^i(\mathbb{R}^{2N+1},\mathcal{N}_2;\pi_i(\Gamma_{N+1}))$, for $1\leq i \leq 2n+1$. By Theorem \ref{bott's theorem}, the group $\pi_i(\Gamma_{N+1})$ vanish for $i \equiv \ 1,3,4,5 \ (mod \ 8)$. Note that $H^i(\mathbb{R}^{2N+1},\mathcal{N}_2;\pi_i(\Gamma_{N+1})) \cong H^{i-1}(M;\pi_i(\Gamma_{N+1})) \cong H_{2n+2-i}(M;\pi_{i}(\Gamma_{N+1}))$. The assumptions that $2n+1$ is of the form $8k+3$ and that the homology groups of $M$ vanish for $i \equiv \ 2,4,5,6 \ (mod \ 8)$ then ensures that the rest of the obstructions also vanish. Theorem \ref{homotopy of contact structure bem} then implies that there is a contact isotopy $\Psi_t$ of $\mathbb{R}^{2N+1}$ relative to $\mathcal{N}_2$ such that $(\Psi_1)_*\cdot (\Phi_1)_*\eta_{ot}= \eta_{ot}$. So, $\Psi_1 \cdot \Phi_1$ gives a contactomorphism of $(\mathbb{R}^{2N+1},\eta_{ot})$ that takes $\iota_1(M)$ to $\iota_2(M)$. 

\end{proof}

\end{document}